\documentclass{article}
\usepackage{latexsym}
\usepackage{color}
\usepackage{amssymb}
\usepackage{amsthm}
\usepackage{amscd,amsfonts}

\usepackage{amssymb}

\usepackage{mathrsfs}

\usepackage{anysize}
\marginsize{2.5cm}{2.5cm}{2.5cm}{3.4cm}

\newtheorem{definition}{Definition}[section]
\newtheorem{proposition}[definition]{Proposition}
\newtheorem{theorem}[definition]{Theorem}
\newtheorem{lemma}[definition]{Lemma}
\newtheorem{corollary}[definition]{Corollary}

\def\Im{\mbox{\rm Im }}
\def\FP{{\cal F}^{\perp}}

\def\Qco{\mathfrak{Qcoh}}

\def\Qcart{QMod_{\rm cart}(R)}
\newcommand{\Ch}{{\mathrm{Ch}}}
\def\cart{Mod_{\rm cart}(R)}
\def\E{{\mathscr E}}

\def\O{{\mathcal O}}
\def\F{{\mathscr F}}

\def\C{{\mathscr C}}

\def\Im{\mbox{\rm Im }}
\def\Ker{\mbox{\rm Ker }}

\def\Z{\mathbb{Z}}
\def\Hom{\mbox{\rm Hom}}
\def\Ext{\mbox{\rm Ext}}

\def\fiz{\leftarrow}

\def\Hom{\mbox{\rm Hom}}
\def\Ext{\mbox{\rm Ext}}

\def\Z{\mathbb{Z}}

\def\E{\mathcal{E}}
\def\F{\mathcal{F}}
\def\C{\mathcal{C}}
\def\A{Mod_{\rm cart}(R)}
\def\FP{{\mathcal{F}}^{\perp}}

\def\Im{\mbox{\rm Im }}
\def\Ker{\mbox{\rm ker}}

\newcommand{\K}{{\mathbf K}}

\def\Qco{\mathfrak{Qco}}

\newcommand{\Flat}{R\!\!\textrm{ -}{\rm Flat}}
\newcommand{\QProj}{{\rm Proj}\ X}
\newcommand{\QFlat}{{\rm Flat}\ X}
\newcommand{\QQProj}{{\rm Proj}\ \mathcal X}
\newcommand{\QQFlat}{{\rm Flat}\ \mathcal X}

\newcommand{\dgC}{{\mathit dg\,}\widetilde{{\mathcal C}}}
\newcommand{\barC}{{\widetilde{{\mathcal C}}}}

\newcommand{\dgF}{{\mathit dg\,}\widetilde{{\mathcal F}}}
\newcommand{\barF}{{\widetilde{{\mathcal F}}}}

\begin{document}

\title{The Derived Category of quasi-coherent sheaves\\ on an Artin stack via model structures}

\bigskip
\author{ 
\  Sergio Estrada{\footnote {This paper has been completed during the author's stay at Max Planck Institute for Mathematics in Bonn. The author would like to thank its hospitality and the excellent conditions provided for his stay at the institution.}}  \\ \ \ sestrada@um.es \\
\ \
Departamento de Matem\'atica Aplicada \\
\ \ \ Universidad de Murcia \\
\ \ 30100 Murcia \\
SPAIN}

\date{}

\maketitle

\bigskip

\thispagestyle{empty}

\begin{abstract}
We define the derived category of quasi--coherent modules for certain Artin stacks as the homotopy category of two Quillen monoidal model structures on the corresponding category of unbounded complexes of quasi--coherent modules.

\end{abstract}

{\footnotesize{\it 2010 Mathematics Subject Classification.}
Primary 14A20; Secondary 55U35,18E30}

{\footnotesize{\it Key words and phrases}: algebraic stack, Deligne-Mumford stack, cartesian $R$-module, Gro\-then\-dieck category,
model category structure.}

\baselineskip=22pt
\section{Introduction}

In \cite{EE} we develop a method for finding a family of generators of the
so-called category of quasi-coherent $R$-modules 
on an arbitrary quiver (cf. \cite[Corollary 3.5]{EE}) and we prove that the
class of flat quasi--coherent
$R$-modules is covering (cf. \cite[Theorem 4.1]{EE}). The first part of the present paper is
devoted to showing that the same
arguments of \cite{EE} can also be used in a much more general setup, that is
that of cartesian $R$-modules on a flat presheaf of rings $R$ over a small
category $\C$. This extends the main application in \cite{EE} to the category
$\Qco(X)$ of quasi--coherent sheaves on a scheme $X$ but also to the
category $\Qco({\mathcal X})$ of quasi--coherent sheaves on an Artin stack or
on a Deligne-Mumford stack. This seems to be known to some authors, but the lack of a published result of this fact in the literature becomes it into an interesting consequence of the results of \cite{EE}.   

In the second part of the paper we deal with the derived category of quasi-coherent sheaves on an Artin stack. From its abstract definition one has little control over the morphisms in the derived category and in fact it is not clear if we have only a set of maps between any two objects. A solution to these questions is provided by defining a good Quillen's theory of model categories in $\Ch( \Qco({\mathcal X}))$ (\cite{Quillen}) where the weak equivalences are the quasi-isomorphisms. Then it follows from Quillen's theory that the corresponding homotopy categoy, the derived category of $\Qco(\mathcal X)$, is truly a category. Moreover there is a simple description of the set of morphisms from two objects $M$ and $N$ in the derived category as chain homotopy classes of chain maps from a cofibrant replacement for $M$ to a fibrant replacemnent of $N$. 
Moreover, the category $\Qco({\mathcal X})$ has a tensor product, which naturally inherits to $\Ch(\Qco(\mathcal X))$, for which we should be able to compute left derived functors, so we would like that our Quillen model category structures to be {\it monoidal}. To achieve this goal, we give a general theorem (Theorem \ref{apmodel2}) that guarantees the existence of cofibrantly generated model category structures in the category of unbounded complexes of cartesian $R$-modules and later in Section \ref{secf} we specialize to the category of quasi-coherent $\O_X$-modules over Artin stacks. This is an extension of the previous papers \cite{G} and \cite{EGPT} from schemes to algebraic stacks. Our main application is that for a geometric stack $\mathcal X$ (see \cite{TV08} and \cite{Lurie}) we show the existence of a flat monoidal model category structure on $\Ch(\Qco(\mathcal X))$ (Theorem \ref{pflat}) and for algebraic stacks that satisfy the resolution property (these include global quotient stacks) we show in Theorem \ref{presol} that there is a locally projective monoidal model structure on $\Ch(\Qco(\mathcal X))$. One immediate implication from \cite{May} is then that the triangulated structure of $\mathbf{D}(\Qco(\mathcal X))$ is strongly compatible with the derived tensor product. We finish Section \ref{secf} with a list of consequences of the results of this paper for categories of modules over a flat Hopf algebroid and on the existence of adjoint functors in homotopy categories for algebraic stacks.

\section{Cartesian modules on quivers}
A quiver $Q$ is a
directed graph. An edge of
a quiver from a vertex $v_1$ to a vertex $v_2$ is denoted by
$a:v_1\to v_2$ or $v_1\stackrel{a}{\to}v_2$, the symbol $E$ will
denote the set of edges. A quiver $Q$ may be thought as a category
in which the objects are the vertices of $Q$ and the morphisms are
the paths (a path is a sequence of edges) of $Q$. The set of all
vertices will be denoted by $V$.

Let $Q=(V,E)$ be a quiver and let $R$ be a presheaf from $Q$
in the category of commutative rings, that is, for each vertex $v\in V$ we
have a ring $R(v)$ and for an edge $a:v\to w$ we have a ring
homomorphism $R(a^{op}):R(w)\to R(v)$.

We shall say that we have an $R$-module $M$ when we have an
$R(v)$-module $M(v)$ and a morphism $M(a^{op}):M(w)\to M(v)$ for each
edge $a:v\to w$ that is $R(v)$-linear. The $R$-module $M$ is said to be a
{\it cartesian $Q$-module} if for each edge $a:v\to w$ as above the morphism
$$R(v)\otimes_{R(w)}M(w)\to M(v)$$ given by $r_v\otimes m_w\mapsto
r_v M(a^{op})(m_w)$, $r_v\in R(v),\ m_w\in M(w)$ is an
$R(v)$-isomorphism.

The category of cartesian $Q$-modules is abelian when $R$ is
such that for an edge $v\to w$, $R(v)$ is a {\sl flat}
$R(w)$-module (so the kernel of a morphism between two cartesian $Q$-modules is
also cartesian). In this case
we say $R$ is flat. Coproducts and colimits may be computed
componentwise so direct limits are exact and, as a result of
Proposition \ref{nocont}, we can find a system of generators in
the category. Therefore the category of cartesian $Q$-modules is indeed a
Grothendieck category when $R$ is flat.

By the tensor product, $M\otimes_R N$, where $M$ is a right
$R$-module and $N$ a left $R$-module, we mean the $\Z$-module
($\Z(v)=\Z$, for all $v\in V$ and $\Z(a)=id_{\Z}$ for all $a\in
E$) such that
$$(M\otimes_R N)(v)=M(v)\otimes_{R(v)} N(v),$$with $(M\otimes_R N)(a)$
the obvious map. We then get the notion of a {\it flat $R$-module} and a {\it flat
cartesian $Q$-module}. We also get the notion of {\it locally projective cartesian $Q$-module} $M$, by defining that $M(v)$ is a projective $R(v)$-module for each $v\in V$. 

Given an arbitrary quiver $Q$ and a flat presheaf of rings $R$ over $Q$, we will
denote by $QMod_{\rm cart}(R)$ the category of cartesian $Q$-modules over $R$. 
\section{Cartesian modules on small categories}

Now let $\C$ be any small category, and let $R$ be a flat presheaf of rings on
$\C$. We will consider the category $\cart$ of cartesian $R$-modules. This is
an abelian category and, as a consequence of Proposition \ref{nocont}, it will
be a Grothendieck category. There is a notion of flat and of locally projective cartesian module as
before. Let $Q$ be the quiver whose vertices are the objects of $\C$, and whose
edges are the morphisms of $\C$. It is then clear that the
category $\cart$ is a full subcategory of the category $\Qcart$. Furthermore it
is also clear that if $M\subseteq N$ in $\Qcart$ and $N$ is a cartesian
$R$-module, then $M$ will be automatically a cartesian $R$-module as well. This
easy observation is
crucial in proving our main result in the next section and giving our main applications.

\section{Generators in $\cart$. Application to algebraic stacks}

With the observations made in the previous sections, we can use Proposition
3.3 of \cite{EE} to infer that $\cart$ is a Grothendieck
category. Throughout this section we will assume that $\C$ is a small category
and $R$ is a flat presheaf of rings on it. We shall denote by $Q$ the
quiver associated to $\C$.

\begin{definition} Let $M$ be a cartesian $Q$-module. The
cardinality of $M$ is defined as the cardinality of the coproduct
(in the category of sets) of all modules associated to the
vertices $v\in V$, that is
$$|M|=|\sqcup_{v\in V}M(v)|$$
\end{definition}

\begin{proposition}\label{nocont}

Let $\C$ be any small category with associated quiver $Q_{\C}=(V,E)$
and $M$ a cartesian $R$-module. Let
$\lambda$ be an infinite cardinal such that $\lambda\geq |R(v)|$ for
all $v$ and such that $\lambda\geq max\{|E|,|V|\}$. Let
$X_v\subseteq M(v)$ be subsets with $|X_v|\leq \lambda$ for all
$v$. Then there is cartesian $R$-submodule $M'\subseteq M$ with
$M'(v)$ pure for all $v$, with $X_v\subseteq M'(v)$ for all $v$
and such that $|M'|\leq \lambda$.
\end{proposition}
{\bf Proof. } The proof of \cite[Proposition 3.3]{EE} gives a cartesian
$Q$-submodule $M'$ of $M$ satisfying the desired properties. But then by the
previous comment, as $M$ is cartesian, $M'$ will be also cartesian $R$-module.
$\Box$.

\medskip
\begin{definition}
A cartesian $R$-submodule $M'$ of an $R$-module $M$ is said
to be pure whenever $M'(v)$ is a pure $R(v)$-submodule of $M(v)$,
for every vertex $v\in V$.
\end{definition}

\medskip

\begin{corollary}\label{cor1}
There exists an
infinite cardinal $\lambda$ such that every
cartesian $R$-module $M$ is the sum of its quasi-coherent
$R$-submodules of type $\lambda$.

\end{corollary} 

{\bf Proof. } Let $M$ be any cartesian $R$-module and take an
element $x\in M$. Then, by Proposition \ref{nocont} we find a
(pure) cartesian $R$-submodule $S_x$ of $M$ with $|S_x|\leq
\lambda$ and $x\in S_x$. Thus $M=\sum_{x\in M} S_x$. $\Box$

\medskip
As a consequence of this we have that $\cart$ is a Grothendieck category
whenever $R$ is a flat presheaf of rings,
for if we take a set $Z$ of representatives of
cartesian modules with cardinality bounded by $\lambda$, it is
immediate that the single cartesian $R$-module $\oplus_{S\in
Z}S$ generates the category of quasi-coherent $R$-modules. 

Now if we focus on particular instances of small categories we have the
following significant consequences. The first one is due to Gabber (cf. \cite[Lemma 2.1.7]{Conrad}). 
\begin{corollary}
Let $(X,\O_X)$ be any arbitrary scheme. Each
quasi-coherent sheaf can be written as a sum of its pure
quasi-coherent subsheaves of type $\lambda$. Thus $\Qco(X)$ is a
Grothendieck category.

\end{corollary}
{\bf Proof.} We let $\C$ consisting of all the affine open $U\subseteq X$. Then
the inclusion between affine open subsets defines a canonical structure of a
partially ordered category on $\C$. Now we let $R$ be the structure sheaf
$\O_X$. Then it is standard that $\cart$ and $\Qco(X)$ are equivalent
categories. So the result will follow from Corollary \ref{cor1}. $\Box$

\medskip\par\noindent
{\bf Remark.}
The previous proof also clarifies a possible misunderstanding on
\cite[Section 2]{EE}. There, the reader may wrongly
think that we are considering the {\it free} category on the affine open
subsets of the scheme $X$ to establish our equivalent category $\mathcal C$.
This is obviously not true, and the gap is easily fixed by saying that we were
assuming the compatibility
condition on our representations there to get the desired equivalence. To be
precise we are just claiming that
$\Qco(X)$ and $\cart$ (or $\C$ in that section) are equivalent.

\medskip\par
Our second application goes back to Artin stacks (cf. \cite{olsson}) and
Deligne-Mumford stacks.

\begin{corollary}
Let $\mathcal X$ be a Deligne-Mumford stack. Then the category $\Qco(\mathcal
X)$ is a Grothendieck category. In particular it is locally presentable and has
arbitrary products.
\end{corollary}
{\bf Proof.} We take $\C$ as the small subcategory of the iso classes of the
category of affine schemes that are \'etale over $\mathcal X$ (such small
subcategory must exist as the iso classes of such schemes form a set, as
\'etale morphisms are of finite type). Then $\cart$ is equivalent to
$\Qco(\mathcal{X})$.
$\Box$

\begin{corollary}\label{esgr}
Let $\mathcal X$ be an algebraic stack with a flat sheaf of rings $\mathcal A$.
Then the category $\Qco(\mathcal X)$ is a Grothendieck category. In particular
it is locally presentable and has arbitrary products.
\end{corollary}
{\bf Proof.} In this case we consider $\C$ to be the category of affine
schemes smooth
over $\mathcal X$ and $R$ as the sheaf of rings $\mathcal A$. Then $\cart$ is
equivalent to the category of quasi--coherent sheaves on $\mathcal X$. $\Box$





 

\section{Preliminaries on cotorsion pairs in $\cart$}\label{Hi}

In this section we present those notions from $\cart$ which will be used in the sequel.
Let $\C$ be any small category, $R$ a flat presheaf of rings on $\C$ and $\cart$ the corresponding category of cartesian $R$-modules.
Let us fix $\lambda$ as in Proposition \ref{nocont}.

We recall that a cartesian $R$-module $M$ is $\kappa$--generated (for $\kappa$ a
regular cardinal) whenever $$\Hom_{\A}(M,-)$$ preserves $\kappa$--filtered
colimits of monomorphisms. Equivalently, $M$ is $\kappa$--generated whenever
$M=\sum_{i\in I}
M_{i}$ is a
$\kappa$--directed union of cartesian submodules, we have $M=M_i$ for some $i\in I$. And
$M$ is $\kappa$--presentable whenever $\Hom_{\A}(M,-)$ preserves
$\kappa$--filtered
colimits. Under our above assumption on $\lambda$ relative to
$\A$, a cartesian $R$-module $M$
is $\lambda$--presentable if
and only if $M$ is $\lambda$--generated and every epimorphism $L\to M$ with
$L$ $\lambda$--generated has a $\lambda$--generated kernel.

Furthermore, given $M\in\cart$ it is easy
to check that the following conditions are equivalent:

\begin{enumerate}
 \item $|M|<\lambda$.
 \item $M$ is $\lambda$--generated.
 \item $M$ is $\lambda$--presentable.

\end{enumerate}

A well-ordered direct system of cartesian $R$-modules, $ (
A_{ \alpha } \mid \alpha
\leq \gamma ) $, is said to be \emph{continuous} if $ A_0  = 0 $ and, for
each limit ordinal $ \beta \leq \gamma $, we have $ \displaystyle
A_{\beta}=\lim_{ \rightarrow } \:  A_{ \alpha } \:$ where the limit is taken
over
all ordinals $\alpha{<}\beta$. A~continuous direct system $ ( A_{ \alpha}
\mid\alpha \leq \gamma ) $ is called a \emph{continuous directed union} if
all morphisms in the system are monomorphisms.

\begin{definition}\label{defilt}
Let $ {\mathcal L} $ be a class of cartesian $R$-modules. An
object $A$ of $\mathcal A$ is \emph{$\mathcal L$-filtered} if  $ A =
\displaystyle \lim_{ \rightarrow } \: A_{ \alpha } $ for a
continuous directed union $ ( A_{ \alpha } \mid  \alpha \leq \gamma
) $ satisfying that,  for each $ \alpha + 1 \leq \gamma
$, $ \mbox{Coker} \: ( A_{ \alpha } \rightarrow A_{ \alpha + 1 } )$
is isomorphic to an element of $ {\mathcal L} $.

\end{definition}

We can easily extend Hill's Lemma (see
\cite[Theorem 4.2.6]{GT}) to $\A$. 
Let $\mathcal J$ be a class of
$\lambda$--presentable cartesian $R$-modules and
let $ M$ be a cartesian $R$-module possesing  a $\mathcal
J$--filtration $\mathcal O = ( M_\alpha \mid \alpha \leq \sigma )$.

By Proposition \ref{nocont}, there exist $\lambda$--presentable
cartesian $R$-submodules $ A_\alpha\subseteq  M_{\alpha+1}$ such
that
$ M_{\alpha +1}= M_\alpha +  A_\alpha$ for each $\alpha
{<}
\sigma$. A set $S\subseteq \sigma$ is called \emph{closed} provided that
$
M_\alpha\cap  A_\alpha \subseteq \sum_{\beta{<}\alpha, \beta\in
S}
A_{\beta}$ for each $\alpha\in S$.

\begin{lemma}\label{HillConcrete} Let $\mathcal H=\{\sum_{\alpha\in S}
A_\alpha
\mid S\textrm{ closed}\,\}$.
Then $\mathcal H$ satisfies the following conditions:
\begin{enumerate}
\item[(H1)] $\mathcal O \subseteq \mathcal H$,
\item[(H2)] $\mathcal H$ is closed under arbitrary sums,
\item[(H3)] $ P/ N$ has a $\mathcal J$--filtration whenever $
N, P \in \mathcal H$ are such that $ N
\subseteq  P$.
\item[(H4)] If $ N \in \mathcal H$ and $ X$ is a $\lambda$--presentable
cartesian $R$-submodules of $M$,
then there exists $ P \in \mathcal H$ such that $ N +  X
\subseteq  P$ and $ P/ N$ is
$\lambda$--presentable.
\end{enumerate}
\end{lemma}

\begin{proof}
Note that for each ordinal $\alpha\leq \sigma$, we have
$
M_\alpha=\sum_{\beta{<}\alpha}  A_{\beta}$, hence $\alpha$ is a closed
subset
of $\sigma$. This proves condition $(1)$. Since any union of closed subsets is
closed, condition $(2)$ holds.

In order to prove condition $(3)$, we consider closed subsets $S,T$ of $\sigma$
such
that $ N=\sum_{\alpha\in S} A_{\alpha}$ and $
P=\sum_{\alpha\in T} A_{\alpha}$. Since $S\cup T$ is closed, we will
w.l.o.g.\ assume that $S\subseteq T$. We define a $\mathcal J$--filtration of
$ P/ N$ as follows. For each $\beta \leq \sigma$, let $
F_{\beta}=(\sum_{\alpha\in T\setminus S,\alpha{<}\beta}
A_{\alpha}+
N)/ N$. Then $ F_{\beta+1}= F_{\beta}+(
A_{\beta}+ N)/ N$ for $\beta\in T\setminus S$ and $
F_{\beta+1}= F_{\beta}$ otherwise.

Let $\beta\in T\setminus S$. Then $  F_{\beta+1}/ F_{\beta}\cong
 A_{\beta}/( A_{\beta}\cap(\sum_{\alpha\in T\setminus
S,\alpha{<}\beta} A_{\alpha}+ N)),$ and since $\beta\in
T\setminus S$
and $T$ is closed, we have
$$ A_{\beta}\cap(\sum_{\alpha\in T\setminus S,\alpha{<}\beta}
A_{\alpha}+ N)= A_{\beta}\cap (\sum_{\alpha\in
S,\alpha{>}\beta} A_{\alpha}+\sum_{\alpha\in T,\alpha{<}\beta}
A_{\alpha})\supseteq$$ $$\supseteq  A_{\beta}\cap (\sum_{\alpha\in
S,\alpha{>}
\beta} A_{\alpha}+( M_{\beta}\cap  A_{\beta}))\supseteq
 M_{\beta}\cap  A_{\beta}.$$
Let $ B_{\beta}=\sum_{\alpha\in S,\alpha{>}\beta}
A_{\alpha}+\sum_{\alpha\in T,\alpha{<}\beta} A_{\alpha}$
We will prove that $ A_{\beta}\cap B_{\beta}=
M_{\beta}\cap
 A_{\beta} $. We have only to show that, $
A_{\beta}\cap  B_{\beta}\subseteq  A_{\beta}\cap
M_{\beta}$. Let $a\in  A_{\beta}\cap  B_{\beta}$. Then
$a=c+a_{\alpha_0}+\cdots +a_{\alpha_k}$ where $c\in \sum_{\alpha\in
T,\alpha{<}\beta} A_{\alpha}\subseteq  M_{\beta}$,
$\alpha_i\in
S$ and $a_{\alpha_i}\in  A_{\alpha_i}$ for all $i\leq k$ and
$\alpha_i{>}\alpha_{i+1}$ for all $i{<}k$. W.l.o.g., we can assume that
$\alpha_0$ is
minimal possible. If $\alpha_0{>}\beta$, then
$a_{\alpha_0}=a-c-a_{\alpha_1}+\cdots
-a_{\alpha_k}\in  M_{\alpha_0}\cap  A_{\alpha_0}\subseteq
\sum_{\alpha\in S,\alpha{<}\alpha_0}  A_{\alpha}$ (since $\alpha_0\in
S$),
in contradiction with the minimality of $\alpha_0$. Since $\beta \notin S$, we
infer
that $\alpha_0{<}\beta$, $a\in  M_{\beta}$, and $
A_{\beta}\cap
 B_{\beta}= A_{\beta}\cap  M_{\beta}$.

So if $\beta \in T\setminus S$ then $ F_{\beta+1}/
F_{\beta}\cong
 A_{\beta}/( M_{\beta}\cap  A_{\beta})\cong 
M_{\beta+1}/ M_{\beta}$, and the latter is isomorphic to an element of
$\mathcal J$ because $\mathcal O$ is a $\mathcal J$--filtration of $ M$.
This finishes the proof of condition $(3)$.

For condition $(4)$ we first claim that each subset of $\sigma$ of
cardinality ${<}\lambda$ is contained in a closed subset of
cardinality ${<}\lambda$. Since $\lambda$ is regular and unions of
closed sets are closed, it suffices to prove the claim only for
one--element subsets of $\sigma$. By induction on $\beta$ we prove
that each $\beta{<}\sigma$ is contained in a closed set $S$ of
cardinality ${<}\lambda$. If $\beta{<}\lambda$ we take $S=\beta+1$.

Otherwise, consider the short exact sequence $0\to 
M_{\beta}\cap  A_{\beta}\to  A_{\beta}\to 
M_{\beta+1}/ M_{\beta}\to 0 $. By our assumption on
$\lambda$, since $A_{\beta}$ is $\lambda$--presentable, so is
$ M_{\beta}\cap  A_{\beta}$. Hence, $ M_{\beta} \cap 
A_{\beta} \subseteq
\sum_{\alpha \in S} A_{\alpha}$ for a subset $S
\subseteq \beta$ of cardinality ${<} \lambda$. By our inductive
premise, the set $ S$ is contained in a closed
subset $S^\prime$ of cardinality ${<} \lambda$. Let $S''=S^\prime \cup
\{\beta\}$. Then $S''$ is closed because $S^\prime$ is closed, and
$ M_{\beta}\cap  A_{\beta}\subseteq \sum_{\alpha \in
S^\prime} A_{\alpha}$.

Finally if $ N=\sum_{\alpha\in S''} A_{\alpha}$ and
$ X$ is a $\lambda$--presentable cartesian submodule of $ M$, then $ X\subseteq
\sum_{\alpha\in
T} A_{\alpha}$ for a subset $T$ of $\sigma$ of cardinality
${<}\lambda$. By the above we can assume that $T$ is closed and put
$ P=\sum_{\alpha \in S'' \cup T} A_{\alpha}$. By (the
proof of) condition $(3)$ $ P/ N$ is $\mathcal
J$--filtered, and the length of the filtration can be taken $\leq
|T\setminus S|{<}\lambda$. This implies that $ P/ N$
is $\lambda$--presentable.
\end{proof}

A {\em cotorsion pair} in a Grothendieck category $\mathcal A$ is a pair of classes of objects of
$\mathcal A$,
$(\mathcal F, \mathcal C)$, such that $\mathcal F = {}^{\perp}\mathcal
C$ and $\mathcal C =
\mathcal F^\perp $, where  $$ {}^{\perp}\mathcal
C=\{F\in {\mathcal A}:\ \Ext^1_{\mathcal A}(F,D)=0,\ \forall D\in C\}$$ and $$
\mathcal F^\perp=\{C\in \mathcal A:\ \Ext^1_{\mathcal A}(G,C)=0,\ \forall G\in\F\}.$$ The
cotorsion pair is said to
be {\em cogenerated by a class of objects $\mathcal{T}$ in $\mathcal A$} if
$\mathcal{T}^\perp = \mathcal C$. When this class
$\mathcal{T}$ is a set and it contains a generator of $\mathcal A$, it is
known that every object $M$ in $\mathcal A$ has {\it enough projectives}, that
is, there exists a short exact sequence $$0\to C\to F\to M\to 0,$$ with $F\in
\F$ and $C\in \C$, and {\it enough injectives}, that is, an exact sequence
$$0\to M\to C'\to F'\to 0$$ with $F'\in \F$ and $C'\in \C$ (see e.g.
\cite[Theorem 2.5]{EEGO2} or \cite[Corollary 6.6]{hovey2}).
A cotorsion pair having enough injectives and projectives is called {\it
complete}. 

By using the previous version of Hill's Lemma, we can get the analogous version of
Kaplansky's Theorem for cotorsion pairs (see \cite[Theorem 4.2.11]{GT}) for $\cart$. Given a class $\F$ in $\A$ we
will denote by $\F^{\kappa}$ the class of all $\kappa$--presentable objects in
$\F$.

\begin{theorem}\label{kaplans}
 
Let $\kappa$ be an uncountable regular cardinal such that $\kappa\geq \lambda$.
Let $(\F,\C)$ be a cotorsion pair in $\A$ such that $\F$ contains a set of $\kappa$-presentable generators of $\A$. Then the following conditions are
equivalent:
\begin{enumerate}
 \item the cotorsion pair $(\F,\C)$ is cogenerated by a class of
$\kappa$--presentable cartesian $R$-modules.
\item Every cartesian $R$-module in $\F$ is $\F^{\kappa}$--filtered. 

\end{enumerate}

\end{theorem}
\begin{proof}
 $(1)\Rightarrow (2)$
We can assume that $(\F,\C)$ is cogenerated by a {\it
set} $\mathcal T$ of $\kappa$--presentable cartesian $R$-modules and that, up to isomorphism, each generator of $\A$ is in $\mathcal T$.
By the arguments
given in the proof of \cite[Lemma 2.4 and Theorem 2.5]{EEGO2}, the class $\F$
consists of all retractions of $\mathcal T $--filtered objects. So the claim
follows with the same proof (adapted to our setting and by using Lemma
\ref{HillConcrete}) of \cite[Lemma 4.2.10]{GT}.

\noindent
$(2)\Rightarrow (1)$ By Eklof's Lemma (\cite[Theorem 1.2]{Eklof}) in his version
for arbitrary Grothendieck categories, an object $D\in\C$ if, and only if,
$D\in (\F^{\kappa})^{\perp}$. So $(\F,\C)$ is cogenerated by $\F^{\kappa}$.
\end{proof}

\medskip\par\noindent
{\bf Remark:} Theorem \ref{kaplans} can be easily extended to $\Ch(\cart)$, the category of {\it unbounded chain complexes} of cartesian $R$-modules.

\section{Complete cotorsion pairs in $\Ch(\cart)$}

We will denote by $\Ch(\A)$ the category of unbounded chain complexes
of $\A$.



We now recall some well-known facts of the category $\Ch(\mathcal A)$
of unbounded chain complexes on an abelian category $\mathcal A$. A complex in $\mathcal A$,
\begin{displaymath}
  \cdots \stackrel{d_{n+2}}{\longrightarrow} X_{n+1}
  \stackrel{d_{n+1}}{\longrightarrow} X_n \stackrel{d_n}{\longrightarrow} X_{n-1}
  \stackrel{d_{n-1}}{\longrightarrow} \cdots,
\end{displaymath}
will be denoted by $(X,d)$, or simply by $X$. And we will denote
by $Z_nX = \ker d_n$, $K_nX = \textrm{Coker } d_n$, $B_nX
= \Im d_{n+1}$ and $H_nX = \frac{Z_nX}{B_nX}$, for every integer
$n$. Given other complex $Y$, $Hom(X,Y)$ will denote the complex
defined by
\begin{displaymath}
  Hom(X,Y)_n = \prod_{k \in \mathbb{Z}}{\rm Hom}_R(X_k,Y_{k+n})
\end{displaymath}
and $\big((f)d^H_n\big)_k  = f_kd_{k+n}^Y-(-1)^nd_k^Xf_{k-1}$
for any $n \in \mathbb{Z}$. The class of all acyclic complexes will be
denoted by $\mathcal E$.

Let us fix a cotorsion pair $(\mathcal{F}, \mathcal{C})$ in
$\A$. We will consider the following subclasses of $\Ch(\A)$:

\begin{enumerate}
\item The class of {\em $\mathcal F$-complexes}, $\barF = \{X \in
  \mathcal{E}:Z_nX \in \mathcal{F}, \ \forall n \in \mathbb{Z}\}$.

\item The class of {\em $\mathcal C$-complexes}, $\barC = \{X \in
  \mathcal{E}:Z_nX \in \mathcal{C}, \ \forall n \in \mathbb{Z}\}$.

\item The class of {\em dg-$\mathcal F$ complexes},
  $$\dgF = \{X \in \Ch(\A): X_n \in
  \mathcal{F} \ \forall n \in \mathbb{Z} \textrm{ and } Hom(X,C)
  \textrm{ is exact}\ \forall C \in \widetilde{\mathcal{C}}\}.$$

\item The class of {\em dg-$\mathcal C$ complexes},
  $$\dgC = \{X \in \Ch(\mathcal A): X_n \in
  \mathcal{C}\ \forall n \in \mathbb{Z} \textrm{ and } Hom(F,X)
  \textrm{ is exact} \ \forall F \in \tilde{\mathcal{F}}\}.$$

\end{enumerate}

We start with the following lemma. We recall that $\lambda$ is fixed as in Proposition \ref{nocont} for $\cart$.
\begin{lemma}\label{exactk} Let $\kappa$ be a regular infinite
cardinal such that
$\kappa\geq \lambda$. Let $ N=(
N^n), M=( M^n)$
be exact complexes such that $ N\subseteq
 M$. For each $n \in \mathbb Z$, let ${ X}_n$ be a $
\kappa$--presentable cartesian $R$-submodule of $ M ^n$. Then there exists an
exact complex $ T=( T^n)$ such that
$ N\subseteq  T\subseteq M$, and for each $n \in \mathbb Z$,
$ T^n\supseteq  N^n+  X_n$, and the object
$ T^n/ N^n$ is $\kappa$--presentable.
\end{lemma}
\begin{proof} (I) First, consider the particular case of $ N = 0$. Let
${ Y}^n_{0} = { X}_n + \delta^{n-1}({ X}_{n-1})$.
Then $({ Y}^n_0)$ is a subcomplex of $ M$.

If $i < \omega$ and ${ Y}^n_i$ is a $\kappa$--presentable cartesian $R$-submodule of
$ M^n$, put ${ Y}^n_{i+1} = { Y}^n_i + { D}^n_i + \delta^{n-1}({ D}^{n-1}_i)$
where ${ D}^n_i$ is a $\kappa$--presentable cartesian submodule of
$ M ^n$ such that $\delta^n ({ D}^n_i) \supseteq Z_{n+1} M
\cap { Y}^{n+1}_i$. (Such ${ D}^n_i$ exists by Proposition \ref{nocont}, since $Z_{n+1} M \cap { Y}^{n+1}_i \subseteq
\hbox{Ker}(\delta^{n+1}) = \mbox{Im}(\delta^n)$.)  Let ${ T}^n = \bigcup_{i
{<} \omega} { Y}^n_i$. Then $Z_{n+1} M \cap { T}^{n+1} =
\bigcup_{i {<} \omega}(Z_{n+1} M \cap { Y}^{n+1}_i) \subseteq
\bigcup_{i {<} \omega} \delta^n ({ Y}^n_{i+1}) \subseteq  \delta^n ({
T}^n)$. It follows that $ T=( T^n)$ is an exact subcomplex of $M$.
By our assumption on $\kappa$, $ T ^n$ is $\kappa$--presentable.

(II) In general, let $\bar{ M} =  M/ N$ and $\bar{
X}_n = ({ X}_n +  N ^n)/ N ^n$. By part (I), there is
an exact complex $\bar{ T}$ such that
$\bar{ T} \subseteq \bar{ M}$, and
for each $n \in \mathbb Z$, $\bar{ T}^n\supseteq \bar{ X}_n$, and
the cartesian $R$-module $\bar{ T}^n$ is $\kappa$--presentable.
Then  $\bar{ T} =  T/ N$ for an exact subcomplex $ N
\subseteq  T \subseteq  M$, and
$ T$ clearly has the required properties.
\end{proof}

We will also need the following Lemma whose proof is similar to the previous
one, so we will omit it.

\begin{lemma}\label{exactk2} Let $\kappa$ be a regular infinite
cardinal such that
$\kappa\geq \lambda$. Let $ N=(
N^n), M=( M^n)$
be complexes such that $ N\subseteq
 M$. For each $n \in \mathbb Z$, let ${ X}_n$ be a 
$\kappa$--presentable cartesian submodule of $ M ^n$. Then there exists a
complex $ T=( T^n)$ such that
$ N\subseteq  T\subseteq M$, and for each $n \in \mathbb Z$,
$ T^n\supseteq  N^n+  X_n$, and the object
$ T^n/ N^n$ is $\kappa$--presentable.
\end{lemma}

\begin{theorem}\label{induc}
Let $(\mathcal F,\mathcal C)$ be a cotorsion pair cogenerated by a set in $\A$  and
such that $\F$ contains a generator of $\cart$. Then the induced pairs
$(\barF,\dgC)$ and $(\dgF,\barC)$ are complete cotorsion pairs.

\end{theorem}
\begin{proof}
 By \cite[Corollary 3.8]{G0} we have induced cotorsion pairs
  $(\barF,\dgC)$ and $(\dgF,\barC)$.  By \cite[Proposition 3.8]{G} and
\cite[Corollary
  6.6]{hovey2} the pair $(\dgF,\barC)$ is complete. Finally to see that the pair
$(\barF,\dgC)$ is complete we will prove that each complex $C\in \barF$ is
$\widetilde{\mathcal F}^{\kappa}$--filtered (for some $\kappa\geq \lambda$), so
the
cotorsion pair $(\barF,\dgC)$ will be cogenerated by a set.
Then the completeness follows from Quillen's small object argument
(see \cite[Corollary
6.6]{hovey2}).

Let $ C=( M^n)\in \barF$. Then for each $n\in
\mathbb{Z}$, $Z_n C\in \mathcal F$ and therefore
$Z_n C$ has an $\mathcal F^{\kappa}$--filtration $\mathcal
O_n=(M^n_{\alpha}\mid \alpha\leq \sigma_n)$. For each
$n\in\mathbb{Z}$, $\alpha < \sigma_n$, consider a~
$\kappa$--presentable cartesian $R$-module $ A_{\alpha}^n
$ such that $M_{\alpha+1}^n=M_{\alpha}^n+A_{\alpha}^n$, and the
corresponding family $\mathcal H_n$ as in Lemma \ref{HillConcrete}. Since
the complex $ C$ is exact, the $\mathcal F^{\kappa}$--filtration
$\mathcal O_{n+1}$ determines a canonical prolongation of $\mathcal
O_n$ into a filtration $\mathcal O'_n=(M_{\alpha}^n\mid \,
\alpha\leq \tau_n)$ of $M^n$ where $\tau_n =  \sigma_n +
\sigma_{n+1}$ (the ordinal sum).

By definition, for each $\alpha \leq \sigma_{n+1}$, $\delta^n$ maps $M_{\sigma_n
+ \alpha}^n$
onto $M_\alpha^{n+1}$. So for each $\alpha < \sigma_{n+1}$ there is a
$\kappa$--presentable cartesian submodule $A^n_{\sigma_n + \alpha}$ of
$M_{\sigma_n + \alpha+1}^n$ such that $\delta^n(A^n_{\sigma_n + \alpha}) =
A^{n+1}_\alpha$.
Since for each $\sigma_n \leq \alpha < \tau_n$ we have $\hbox{Ker}(\delta^n)
\subseteq M_{\alpha}^n$,
it follows that $M_{\alpha+1}^n=M_{\alpha}^n+A_{\alpha}^n$.

Let $\mathcal H'_n$ be the family corresponding to $A^n_\alpha$ ($\alpha <
\tau_n$) by Lemma \ref{HillConcrete}. Since each closed subset of $\sigma_n$ is
also closed when considered as a
subset of $\tau_n$, we have $\mathcal H_n\subseteq \mathcal H_n'$.
Note that, by \cite[Theorem 1.2]{Eklof}, the class of $\F^ {\kappa}$-filtered cartesian $R$-modules is contained in $\F$, so $\mathcal
H_n'\subseteq \mathcal F$ by condition (H3) of Lemma \ref{HillConcrete}.

Notice that $Z_n C = M ^n_{\sigma_n} = \sum_{\alpha < \sigma_n}A_{\alpha}^n$.
We claim that for each closed subset $S\subseteq \tau_n$, we have
$Z_n C \cap \sum_{\alpha\in S}A_{\alpha}^n=
\sum_{\alpha \in S\cap \sigma_n}A_{\alpha}^n\in \mathcal H_n$. To see this, we
first
show that $\sum_{\alpha < \sigma_n}A_{\alpha}^n\cap \sum_{\alpha\in
S}A_{\alpha}^n = \sum_{\alpha \in S\cap \sigma_n}A_{\alpha}^n$. The inclusion
$\supseteq$ is clear, so consider $a\in
(\sum_{\alpha < \sigma_n}A_{\alpha}^n)\cap \sum_{\alpha\in S}A_{\alpha}^n$. Then
$a=a_{\alpha_0}+\cdots +a_{\alpha_k}$ where
$\alpha_i\in S$, $a_{\alpha_i}\in  A_{\alpha_i}^n$ for all $i\leq k$, and
$\alpha_i > \alpha_{i+1}$ for all $i < k$. W.l.o.g., we can assume that
$\alpha_0$ is
minimal possible. If $\alpha_0\geq\sigma_n$, then $a_{\alpha_0}=a-a_{\alpha_1}-
\cdots -a_{\alpha_k}\in (\sum_{\alpha < \alpha_0}A_{\alpha}^n) \cap
A_{\alpha_0}^n\subseteq \sum_{\alpha\in S,\alpha < \alpha_0}
A_{\alpha}^n$ as $\alpha_0\in S$ and $S$ is closed, in contradiction with the
minimality of $\alpha_0$. Hence $\alpha_0 < \sigma_n$, and $a\in \sum_{\alpha
\in
S\cap \sigma_n}A_{\alpha}^n$. So $Z_n C \cap \sum_{\alpha\in S}A_{\alpha}^n =
\sum_{\alpha \in S\cap \sigma_n}A_{\alpha}^n$, and the latter cartesian $R$-module is in $\mathcal H_n$ because $S \cap \sigma_n$ is closed in
$\sigma_n$. This proves our claim.

By induction on $\alpha$, we will construct an $\widetilde{\mathcal
F}^{\kappa}$--filtration
$( C_{\alpha}\mid \alpha\leq \sigma)$ of $ C$ such that $
C_{\alpha}=( N_{\alpha}^n)$, $Z_n C_{\alpha}\in \mathcal H_n$ and
$ N_{\alpha}^n\in \mathcal H'_n$ for each $n \in \mathbb Z$.

First, $ C_0=0$, and if $ C_{\alpha}$ is defined and $
C_{\alpha}\neq  C$, then for each $n\in \mathbb{Z}$ we take a $
\kappa$--presentable object $ X_n$ such that $
X_n\nsubseteq  N_{\alpha}^n$ in case $ N_{\alpha}^n \subsetneq M^n$
(this is possible by our assumption on $\kappa$), or $ X_n=0$ if $M^n = 
N_{\alpha}^n$.
If $M^n =  N_{\alpha}^n$ for all $n \in \mathbb Z$, we let $\sigma = \alpha$
and finish our construction.

By Lemma \ref{exactk} there exists an exact subcomplex $ T=( T^n)$
of $ C$ containing $ C_{\alpha}$ such that for each $n \in \mathbb
Z$, $ T^n\supseteq  N _\alpha^n+  X_n$, and the
cartesian $R$-module $ T^n/ N _\alpha^n$ is $
\kappa$--presentable. Then $ Y_n= T^n=  N_{\alpha}^n+
X'_n$ for a $\kappa$--presentable cartesian $R$-submodule $
X'_n$ of $ M^n$. By condition (H4) of Lemma \ref{HillConcrete} (for $
N= N^n_{\alpha}$ and $ X =  X_n'$), there exists a
object $ Y'_n= P_n$ in $\mathcal H_n'$ such that
$ N_{\alpha}^n+ X_n'= T^n\subseteq  P_n$ and
$ P_n/ N^n_{\alpha}$ is $\kappa$--presentable. Iterating
this process we obtain a countable chain $ Y_n\subseteq
Y'_n\subseteq  Y''_n\subseteq \ldots$ whose union $ N^n_{\alpha +1} \in
\mathcal H'_n$ by condition (H2) of Lemma \ref{HillConcrete}. Then $ C_{\alpha
+1} =
( N^n_{\alpha +1})$ is an exact subcomplex of $ C$ containing $ C_{\alpha}$.
Since $ N^n_{\alpha+1}\in \mathcal H'_n$, we have $Z_n C_{\alpha+1}
= Z_n C \cap  N^n_{\alpha+1} \in \mathcal H_n$ by the claim above.

In order to prove that $ C_{\alpha +1}/ C_{\alpha} \in \widetilde{\mathcal
F}^{\kappa}$, it remains to show
that for each $n \in \mathbb Z$, $Z_n( C_{\alpha +1}/ C_{\alpha})\in \mathcal
F$.
Since the complex $ C_{\alpha +1}/ C_{\alpha}$ is exact, it suffices to prove
that
$F = (\delta^n ( N^n_{\alpha +1}) +  N^{n+1}_{\alpha})/ N^{n+1}_{\alpha} \in
\mathcal F$.

We have $ N^n_{\alpha +1} = \sum_{\alpha\in S}A_{\alpha}^n$ where w.l.o.g., $S$
is a closed subset of $\tau_n$ containing $\sigma_n$. Let $S^\prime = \{ \alpha
< \sigma_{n+1} \mid \sigma_{n} + \alpha \in S \}$. Then $S^\prime$ is a closed
subset on $\tau_{n+1} = \sigma_{n+1} + \sigma_{n+2}$. Indeed, for each $\alpha
\in S^\prime$, we have
$$\sum_{\beta < \alpha}A_{\beta}^{n+1}\cap A_{\alpha}^{n+1} =
\delta^n(\sum_{\beta < \sigma_n + \alpha} A_{\beta}^{n}) \cap
\delta^n(A_{\sigma_n + \alpha}^{n}) $$ $$\subseteq \delta^n(\sum_{\beta <
\sigma_n + \alpha, \beta \in S} A_{\beta}^{n}) =\sum_{\beta < \alpha, \beta \in
S^\prime} A_{\alpha}^{n+1}$$
where the inclusion $\subseteq$ holds because $S$ is closed in $\tau_{n}$ and
$\hbox{Ker}(\delta ^n) \subseteq \sum_{\beta < \sigma_n + \alpha}
A_{\beta}^{n}$.

Since $\delta^n ( N^n_{\alpha +1}) = \sum_{\beta \in S^\prime} A_{\beta}^{n+1}$,
and $ N^{n+1}_{\alpha} = \sum_{\beta \in T} A_{\beta}^{n+1}$ for a closed subset
$T$ of $\tau_{n+1}$,
we have $F = \sum_{\beta \in S^\prime \cup T} A_{\beta}^{n+1}/\sum_{\beta \in T}
A_{\beta}^{n+1}$, so $F \in \mathcal F$ by condition (H3) of Lemma
\ref{HillConcrete} for $\mathcal H^\prime_{n+1}$. This finishes the proof of $
C_{\alpha
+1}/ C_{\alpha} \in \widetilde{\mathcal F}^{\kappa}$.

If $\alpha$ is a limit ordinal we define $ C_{\alpha}=\bigcup_{\beta < \alpha}
C_{\beta}
= ( N_{\alpha}^n)$. Then $ N_{\alpha}^n\in \mathcal H'_n$ by condition (H2) of
Lemma \ref{HillConcrete}, and $Z_n C_{\alpha} = Z_n C \cap  N^n_{\alpha}
\in \mathcal H_n$ by the claim above. This finishes the construction of the
$\widetilde{\mathcal F}^{\kappa}$--filtration of $ C$.

\end{proof}

\section{Model category structures on $\Ch(\cart)$}

In this section we will see that the induced cotorsion pairs in Theorem
\ref{induc} give rise to an abelian model structure in $\Ch(\cart)$ in which the
trivial objects are the exact complexes. According to Hovey in \cite{hovey2} we
have to prove that we have induced complete cotorsion pairs of the form
$(\mathcal C,\mathcal D\cap \E)$ and $(\mathcal C\cap \E, \mathcal D)$, where
$\mathcal C$ and $\mathcal D$ are classes of complexes in $\Ch(\A)$. Following
\cite{hovey2}, if we have two complete cotorsion pairs $(\mathcal C, \mathcal
C')$ and $(\mathcal D',\mathcal D)$ in $\Ch(\A)$, we will say that they are
{\it compatible} if $\mathcal D'=\mathcal C\cap \E$ and $\mathcal C'=\mathcal
D\cap \E$.

\begin{lemma}\label{compat}
Let $(\F,\C)$ be a cotorsion pair in $\cart$ cogenerated by a set, such that $\F$ contains a generator of $\cart$ and  $\F$ is closed under taking kernels of
epimorphisms. Then the pairs $(\barF,\dgC)$ and
$(\dgF,\barC)$ are compatible.

\end{lemma}
\begin{proof}

According to \cite[Corollary 3.9(3)]{G} we have just
to check that $$\Ext_{\A}^n(F,C)=0$$ for any $n>0$ and any $F\in \F$ and $C\in
\C$. By definition of cotorsion pair, $\Ext^1_{\A}(F,C)=0$. Given any exact sequence $$0\to C\to
M\to N\to F\to 0$$ representing an element in $\Ext^2_{\A}(F,C)$, we can
construct the sequence $$0\to C\to P\to G\to F\to 0,$$ which represents the same
element in $\Ext^2_{\A}(F,C)$ but with $G\in \F$ (since $\F$ contains a
generator of $\A$) and $Z=\Ker(G\to F)\in \F$ (because $\F$ is closed under
kernels of epimorphisms). So $\Ext^1_{\A}(Z,C)=0,$ and therefore
$0\to C\to P\to Z\to 0$ splits. But this means that $0\to C\to P\to G\to F\to 0$
(and therefore also $0\to C\to M\to N\to F\to 0$) represents the zero element in
$\Ext^2_{\A}(F,C).$  Proceeding inductively in this way we get our claim.
\end{proof}

\begin{theorem}\label{apmodel2}
 Let $(\F,\C)$ be a cotorsion pair in $\A$ cogenerated by a set and such that
$\F$ contains a generator of $\A$ and $\F$ is closed under kernels of
epimorphisms. The compatible cotorsion
pairs $(\dgF,\barC)$ and $(\barF, \dgC)$ induce an
abelian model category structure in $\Ch(\A)$. In this abelian model structure
the weak equivalences are the homology isomorphisms, cofibrations
(resp. trivial cofibrations) are monomorphisms whose cokernels are in
$\dgF$, (resp. trivial cofibrations are
monomorphisms with cokerneles in $\barF$)  and
fibrations (resp. trivial fibrations) are epimorphisms whose kernels are in 
$\dgC$ (resp. trivial fibrations are
epimorphisms with kernels in $\barC$). The corresponding homotopy category to this model structure is $\mathbf D( \cart)$, the derived category of $\cart$.
\end{theorem}

\begin{proof}
The proof of both theorems is a consequence of Theorem \ref{induc}, Lemma
\ref{compat} and \cite[Theorem
2.2]{hovey2}.

\end{proof}

\section{The Derived category of quasi-coherent sheaves on an Artin stack}\label{secf}

Throughout this section by an algebraic stack we mean an Artin stack with separated and quasi-compact diagonal in the sense of \cite{LMB00}.

An algebraic stack is {\it geometric} if it is quasi-compact with affine diagonal (this definition is due to To\"en and Vezzosi in \cite{TV08} and Lurie \cite{Lurie}). For geometric stacks, Gross in his thesis \cite[(3.5.5)Theorem]{Gro10} proves in a very elegant way that $\Qco(\mathcal X)$ admits enough flat objects (that is, every quasi-coherent $\mathcal O_{\mathcal X}$-module is a quotient of a flat quasi-coherent $\mathcal O_{\mathcal X}$-module. The corresponding result for schemes was known from \cite{AJL1} by using the derived category of quasi-coherent sheaves. We recall that if $\F$ denotes the class of flat quasi-coherent $\mathcal O_{\mathcal X}$-modules, then the elements in $\FP$ are known as {\it cotorsion} quasi-coherent $\mathcal O_{\mathcal X}$-modules.

\begin{theorem}\label{pflat}
Let $\mathcal X$ be an algebraic stack with enough flats (for instance a geometric stack). There is a monoidal model category structure on $\Ch(\Qco(\mathcal X))$ where weak equivalences are homology isomorphisms, the cofibrations (resp. trivial cofibrations) are the monomorphisms whose cokernels are dg-flat complexes (res. flat complexes). The fibrations (resp. trivial fibrations) are the epimorphisms whose kernels are dg-cotorsion complexes (resp. cotorsion complexes). The associated homotopy category is $\mathbf D(\Qco(\mathcal X))$, the derived category of $\Qco(\mathcal X)$.

\end{theorem}

\begin{proof}
Let $\C$ be the category of affine schemes smooth over $\mathcal X$ and $R$ the sheaf of rings $\mathcal O_{\mathcal X}$. We will apply Theorem \ref{apmodel2} to the class $\mathcal F$ of flat cartesian $R$-modules. The fact that $(\F,\FP)$ is a cotorsion pair can be easily derived from Proposition \ref{nocont} and the small object argument \cite[Theorem 2.1.14]{Hov}. Namely, by \cite[Theorem 4.1]{EE} for each $M\in \cart$ there exists a short exact sequence $0\to C\to F\to M\to 0$ with $C\in \FP$ and $F\in \F$. Now if $M\in ^{\perp}\!\!(\FP)$ the sequence splits and so $M$ is also flat cartesian $R$-module. Hence $(\F,\FP)$ is a cotorsion pair. By Proposition \ref{nocont} we see that each $F\in \F$ is $\F^{\kappa}$-filtered for certain $\kappa\geq \lambda$. So $\F$ is cogenerated by the set of iso classes of $\kappa$-presentable objects in $\F$. 

Finally to get that the model structure is monoidal we apply \cite[Theorem
5.1]{G} by observing that the class $\mathcal F$ satisfies conditions (1), (2) and (3) of that Theorem.
\end{proof}
For the next application we need to recall the definition of (infinite dimensional) vector bundle (cf. \cite[Section 2, Definition]{D}). A quasi-coherent $\O_{\mathcal X}$-module $M$ is an {\it infinite dimensional vector bundle} if it is locally projective. Let $\F$ be the class of all infinite dimensional vector bundles. An algebraic stack $\mathcal X$ has the {\it resolution property} if every quasi-coherent sheaf is a quotient of a filtered direct limit of locally free sheaves of finite type. In particular if an algebraic stack satisfies the resolution property, the class $\F$ contains a family of generators for $\Qco(\mathcal X)$, in other words $\Qco(\mathcal X)$ has enough vector bundles. By \cite[Theorem 1.1]{Totaro} for a normal noetherian algebraic stack $\mathcal X$ the resolution property is equivalent to $\mathcal X$ being isomorphic to the quotient stack of some quasi-affine scheme by an action of the group GL$_n$. This has been extended by \cite[Theorem 6.3.1]{Gro10} for non-normal noetherian stacks, and by \cite{Rydh} without noetherian hypothesis.

\begin{theorem}\label{presol}
Let $\mathcal X$ be an algebraic stack with pointwise affine stabilizer group that satisfies the resolution property (for instance if $\mathcal X$ is a global quotient stack, cf. \cite[2.18]{Tho87}). There is a monoidal model category structure on $\Ch(\Qco(\mathcal X))$ where weak equivalences are homology isomorphisms, the cofibrations (resp. trivial cofibrations) are the monomorphisms whose cokernels are dg-complexes of (infinite dimensional) vector bundles (resp. exact complexes of vector bundles whose every quasicoherent sheaf of cycles is a vector bundle). The fibrations (resp. trivial fibrations) are the epimorphisms whose kernels are dg-orthogonal to the class of infinite dimensional vector bundles. The associated homotopy category is $\mathbf D(\Qco(\mathcal X))$, the derived category of $\Qco(\mathcal X)$.

\end{theorem}
\begin{proof}
Again we let $\C$ be the category of affine schemes smooth over $\mathcal X$ and $R$ the sheaf of rings $\mathcal O_{\mathcal X}$.
Then we will apply Theorem \ref{apmodel2} to the class $\mathcal F$ of locally projective cartesian $R$-modules. To show that $(\F,\FP)$ is a cotorsion pair cogenerated by a set we will first prove that every $F\in \F$ is $\F^{\kappa}$-filtered for certain $\kappa\geq\lambda$. By Kaplansky Theorem each $F(v)$ is a direct sum of countably generated projective $R(v)$-modules, in particular every $F(v)$ possess a filtration $\mathcal P_v$ by countably generated projective $R(v)$-modules.  Let $\mathcal{H}_v$ be  the family
associated to $\mathcal{P}_v$ by Lemma \ref{HillConcrete} and
$\{x_{v,\alpha} | \alpha < \tau _v \}$ be an $R(v)$-generating set
of the $R(v)$-module $F(v)$. W.l.o.g., we can
assume that for some ordinal $\tau$, $\tau = \tau _v$ for all $v$.

Let us denote by $(F_\alpha |\
\alpha \leq \tau)$ the desired $\mathcal F^ {\kappa}$-filtration that we will get for $F$. Let
$F_0=0$. Assume that $F_\alpha$ is defined for some
$\alpha < \tau$ such  that $F_{\alpha}(v) \in \mathcal{H}_v$ and
$x_{v,\beta} \in F_{\alpha}(v)$ for all $\beta < \alpha$ and all $v$. Set  $N_{v,0}= F_\alpha (v)$. By condition (H4) of
Lemma \ref{HillConcrete}, there is a module $N_{v,1} \in \mathcal{H}_v$ such that
$N_{v,0} \subseteq N_{v,1}$ and $N_{v,1} / N_{v,0}$ is $\lambda$-presentable.

By Proposition \ref{nocont} there is a
cartesian $R$-submodule $G_1 $ of $F$
such that $F_\alpha \subseteq G_1 $ and $G_1
/ F_\alpha$ is $\lambda$-presentable. Therefore $G_1 (v) =
N_{v,1}+ \langle S_v \rangle$ for a set  $S_v \subseteq
G_1  (v)$, of cardinality $<\lambda$. Now by condition (H4) of Lemma \ref{HillConcrete}, there is
a module $N_{v,2} \in \mathcal{H}_v$ such that $G_1  (v)= N_{v,1} +
\langle S_v \rangle \subseteq N_{v,2}$ and $N_{v,2} / N_{v,1}$ is $\lambda$-presentable. Following in this manner, we get
a countable chain $(G_n  |\
n < \aleph_0)$ of cartesian $R$-submodules of $F$ and a countable chain $(N_{v,n} | n < \aleph_0)$ of
$R(v)$-submodules of $F(v)$. We define $F_{\alpha+1}= \sum_{n < \aleph_0} G_n $. Then it is clear that $F_{\alpha+1}\subseteq F$ is a cartesian $R$-module satisfying
$F_{\alpha+1}(v)= \sum_{n < \aleph_0} G_n (v)$ for each $v$. By condition (H2) of
Lemma \ref{HillConcrete}, we deduce that $F_{\alpha+1}(v) \in \mathcal{H}_v$
and $|F_{\alpha+1}(v) / F_\alpha (v)|<\lambda$. Therefore $F_{\alpha+1} / F_\alpha \in \mathcal{F}^ {\kappa}$.
Finally if $\alpha\leq \tau$ is an ordinal limit and $F_\beta$ is defined for every $\beta < \alpha$, we define
$F_\alpha = \sum_{\beta < \alpha} F_\beta$.
Since $x_{v,\alpha} \in F_{\alpha+1}(v)$ for each $v$ and $\alpha <
\tau$, we have  $F_\tau(v)=F(v)$. So $(F_\alpha |\ \alpha
\leq \tau)$ is the desired $\mathcal{F}^{\kappa}$-filtration of $F$. Then, by \cite[Theorem 1.2]{Eklof} the pair $(\F,\FP)$ is cogenerated by the set of iso classes of $\mathcal{F}^{\kappa}$ cartesian $R$-modules.

By \cite[Lemma 2.4, Theorem 2.5]{EEGO2}, for all $A\in \cart$ there
exists a
short exact sequence\begin{equation}\label{equat1}
0\to  A\to P\to Z\to 0
\end{equation}
where $P\in \FP$ and
$Z$ has an $\F$-filtration. Given any $M\in\cart$, since $\mathcal X$ satisfies the resolution property,
there
exists a short
exact sequence
$$0\to U\to G\to M\to 0 $$ where  $G$ is a
direct sum of locally free cartesian $R$-modules (of finite type).
Now let
$$0\to U\to N\to  Z\to 0$$ be exact with $
N\in \FP$ and $Z$ admitting an $\F$-filtration. Form a pushout
and
get

$$\begin{CD}
@.     0@.        0      @.     @.\\
@.     @VVV       @VVV   @.     @.\\
0@>>>  { U}@>>> { G }@>>>  { M}@>>>  0\\
@.     @VVV  @VVV   @|     @.\\
0@>>>  { N}@>>> { Y}@>>>  { M}@>>>  0\\
@.     @VVV       @VVV @.  @.\\
@.     { Z} @= {Z} @.  @.\\
@.     @VVV       @VVV   @.     @.\\
@.     0@.        0      @.     @.
\end{CD}$$
Then since $G\in \F$ 
and $Z$ has an $\F$-filtration (so $Z\in\mathcal \F$, since $\F$ is closed under $\F$-filtrations), we see that $Y\in \mathcal F$. Also $N\in \FP$.
Hence if $M\in  ^{\perp}\!\!(\FP)$ we get that $0\to N\to Y\to
M\to 0$ splits and
so $ M$ is a direct summand of $Y\in \mathcal F$. But then
$M\in \mathcal F$ because $\mathcal F$ is closed under direct summands.

\medskip\par

Again, the model structure is monoidal because the class of locally projective cartesian $R$-modules is contained in the class of flat cartesian $R$-modules, hence condition (1) of \cite[Theorem
5.1]{G} holds (and so in particular $\F$ contains the unit $R$ for the monoidal structure on $\Qco(\mathcal X)$, so condition (3) holds). It is also immediate to notice that the tensor product of two locally projective cartesian $R$-modules is again locally projective. So condition (2) of \cite[Theorem 5.1]{G} is satisfied, what finishes the proof.

\end{proof}

{\bf Remarks:}
\begin{enumerate}
\item There is a slightly different notion of algebraic stack in the literature, due to Goerss \cite{Goe04} by considering that the diagonal morphism is affine and that there is an affine scheme $U$ and a faithfully flat 1-morphism, $p:U\to \mathcal X$. Given a flat Hopf algebroid in \cite[Section 3]{Nau07} it is shown that there is an equivalence between the category of (left) comodules on a flat Hopf algebroid and the category of $\O_{\mathcal X}$-modules on a certain algebraic stack, in the sense of Goerss (see \cite{Holl} for a generalization). The results of this paper apply to this setting, just by picking the suitable category $\cart$, thus providing monoidal model category structures on $\Ch(\Gamma)$, the category of unbounded complexes of (left) $\Gamma$-comodules over $(A, \Gamma)$. In case of "well-behaved" flat Hopf algebroids (see \cite{Hov3} for a precise formulation of "well-behaved") Hovey in \cite{Hov3} already defined and studied a model category structure over $\Ch(\Gamma)$. Its associated model category is the derived category of $(A,\Gamma)$. 

\item Let $X$ be a quasi-compact semi-separated scheme. The derived category of quasi-coherent sheaves is a stable homotopy category in the sense of \cite{HPS97}.
This was shown in \cite{AlonsoT}. The main ingredients of their proof where the facts that $\Qco(X)$ is a Grothendieck category (for any scheme $X$) and that each quasi-coherent $\O_X$-module admits flat resolutions provided that $X$ is quasi-compact and semi-separated. For the category $\Qco(\mathcal X)$ ($\mathcal X$ an arbitrary Artin stack) the axioms (a) and (d) of \cite[Definition 1.1.4]{HPS97} trivially hold. By Corollary \ref{esgr} $\Qco(\mathcal X)$ is Grothendieck, hence each cohomology functor on $\Qco(\mathcal X)$ is representable (cf. \cite[Theorem 5.8]{AJS1}), so the axiom (e) of \cite{HPS97} holds. Now it seems reasonable to conjecture that using the results of this paper and \cite[(3.5.5) Theorem]{Gro10} the remainder axioms (b) and (c) of \cite[Definition 1.1.4]{HPS97} may be followed as in \cite{AlonsoT} for the case of (quasi-compact and semi-separated) schemes.

\item Let $\QQFlat$ be the class of flat quasi-coherent $\O_X$-modules over a geometric stack $\mathcal X$. Again by \cite[(3.5.5) Theorem]{Gro10} and the application of the small object argument (cf. \cite[Theorem 4.1]{EE}) the pair $(\QQFlat,\QQFlat^{\perp})$ is a complete cotorsion pair.

In his thesis \cite{Mur2}, Murfet defines the {\it mock homotopy category of projectives},
$\K_m(\QProj)$ for a quasi--compact and semi--separated scheme $X$. The starting point is
Neeman's description (cf. \cite[Facts 2.14(iii)]{neeman1}) on the affine case $X=Spec(R)$ of $\K(\QProj)$ as the Verdier quotient
$$\K(\Flat)/\K(\widetilde{\Flat})$$ and define $\K_m(\QProj)$ as the corresponding
Verdier quotient $\K(\QFlat)/\K(\widetilde{\QFlat})$ for arbitrary
quasi--compact and semi--separated schemes (see \cite[Definition 3.3 and Proposition
3.4]{Mur2}). Then he proves the existence of a right adjoint functor of the
Verdier quotient map $j^*:\K(\QFlat)\to \K_m(\QProj)$, thus extending to quasi--compact and semi--separated schemes the affine case \cite[Theorem 0.1]{neeman}. Now it is a general fact that any complete cotorsion pair $(\mathcal A,\mathcal B)$ in $\Ch(\cart)$ enables to show the existence of right (resp. left) adjoints of the embeddings $\K(\mathcal A)\to \K(\A)$ and $\K(\mathcal B)\to
\K(\A)$, provided that $\mathcal A$ is closed
under taking suspensions (see \cite[Theorem 3.5]{EBIJ}). Now by Theorem \ref{induc} the pair $(\QQFlat,\QQFlat^{\perp})$ in $\Qco(\mathcal X)$ gives rise to the complete cotorsion pair $(\barF,\dgC)$ in $\Ch(\Qco(\mathcal X))$ (where $\mathcal F=\QQFlat$ and $\C=\QQFlat^{\perp}$).  Hence we can extend Murfet and Neeman results to geometric stacks, to conclude that for a geometric stack $\mathcal X$ the canonical map $j^*:\K(\QQFlat)\to \K_m(\QQProj)$ has a right adjoint functor.
As a consequence there is a localization sequence $$\K(\widetilde{\QQFlat})\to
\K(\QQFlat)\to \K_m(\QQProj).$$

\end{enumerate}


\begin{thebibliography}{}

\bibitem[ATJLL97]{AJL1}
{\sc L.\ Alonso Tarr\'{\i}o, A.\ Jerem\'{\i}as L\'opez, J.\ Lipman},
{\sl Local homology and cohomology on schemes},  Ann. Sci. \'Ecole Norm.
Sup. {\bf 30(4)} (1997), 1--39.


\bibitem[AJPV08]{AlonsoT}
{\sc L.\ Alonso Tarr\'{\i}o, A.\ Jerem\'{\i}as L\'opez, M.\ P\'erez Rodr\'{\i}guez and M.\ J.\ Vale
Gonsalves}, {\sl The derived category of quasi-coherent sheaves and axiomatic stable homotopy}, Adv.\ in Math. {\bf 218}(2008), 1224--1252.


\bibitem[ATJLSS00]{AJS1}
{\sc L.\ Alonso Tarr\'{\i}o, A.\ Jerem\'{\i}as L\'opez and M.\ J.\ Souto Salorio}, {\sl Localization in categories of complexes and unbounded resolutions}, Canad.\ J.\ Math. {\bf 52}(2000), 225--247.



\bibitem[EBIJR12]{EBIJ}
{\sc D. Bravo, E. Enochs, A. Iacob, O. Jenda and J. Rada}, {\sl Cotorsion pairs
in C(R-Mod)}, Rocky\ Mountain\ J.\ Math. {\bf 42} (2012), 1787--1802.


\bibitem[Con00]{Conrad}
{\sc B.\ Conrad}, {\sl Grothendieck duality and base change}, Lecture
Notes in Mathematics, {\bf 1750}, Springer-Verlag, 2000.









\bibitem[Dri06]{D} {\sc V.\ Drinfeld},
{\sl Infinite--dimensional vector bundles in algebraic geometry: an
introduction}, in 'The Unity of Mathematics', Birkh\&quot; auser, Boston
2006, pp.\ 263--304.


\bibitem[Ekl77]{Eklof} {\sc P.C.\ Eklof},
{\sl Homological algebra and set theory}, Trans.\ Amer.\ Math.\ Soc.\
{\bf 227}(1977), 207--225.





\bibitem[EE05]{EE} {\sc E.\ Enochs and S.\ Estrada}, {\sl Relative homological
algebra in the category of quasi-coherent sheaves}, Adv.\ in Math.
{\bf 194}(2005), 284--295.

\bibitem[EE12]{EE12} {\sc E.\ Enochs and S.\ Estrada}, {\sl Cartesian modules in small categories},  preprint, \texttt{arXiv:math.AG/1203.5724}.




\bibitem[EEGRO04]{EEGO2}
{\sc E.\ Enochs, S. Estrada, J.R. Garc\'{\i}a Rozas, and L. Oyonarte}, {\sl
Flat covers in the category of quasi-coherent sheaves over the
projective line.} Comm.\ Algebra. {\bf 32}(2004), 1497--1508.







\bibitem[EGPT12]{EGPT}
{\sc S.\ Estrada, P.\ Guil Asensio, M.\ Prest and J.\ Trlifaj},
{\sl Model category structures arising from Drinfeld vector bundles}, Adv.\ in Math.
{\bf 231}(2012), 1417--1438.


\bibitem[Gil04]{G0} {\sc J.\ Gillespie},
{\sl The flat model structure on Ch(R)}, Trans.\ Amer.\ Math.\ Soc.\ {\bf
356}(2004), 3369--3390.

\bibitem[Gil07]{G} {\sc J.\ Gillespie},
{\sl Kaplansky classes and derived categories}, Math.\ Z. {\bf
257}(2007), 811--843.

\bibitem[Gil08]{G3} {\sc J.\ Gillespie},
{\sl Cotorsion pairs and degreewise homological model structures}, Homol.\
Homot.\ App.\ {\bf 10}(2008), 283--304.





\bibitem[GT06]{GT} {\sc R.\ G\" obel and J.\ Trlifaj},
{\sl Approximations and Endomorphism Algebras of Modules}, GEM {\bf
41}, W.\ de Gruyter, Berlin 2006.


\bibitem[Goe04]{Goe04}
{\sc P. Goerss}, {\sl (Pre-)sheaves of ring spectra over the moduli stack of formal group laws. Axiomatic, Enriched and
Motivic Homotopy Theory}, 101-131, NATO Sci. Ser. II Math. Phys. Chem., {\bf 131}, Kluwer Acad. Publ., Dordrecht,
2004.



\bibitem[Gro10]{Gro10} {\sc P.\ Gross},
{\sl Vector bundles as generators on schemes and stacks}, PhD. Thesis, D\"usseldorf, May 2010.







\bibitem[Holl08]{Holl} {\sc S.\ Hollander}, {\sl Characterizing algebraic stacks}, Proc. Amer. Math. Soc.\ {\bf 136} (2008), 1465--1476.



\bibitem[Hov98]{Hov}
{\sc M. Hovey}, {\sl Model Categories}, MSM {\bf 63},
Amer.\ Math.\ Soc., Providence, RI, 1998.

\bibitem[Hov01]{Hov2}
{\sc M. Hovey}, {\sl Model category structures on chain complexes of
sheaves}, Trans.\ Amer.\ Math.\ Soc. {\bf 353(6)}(2001), 2441--2457.

\bibitem[Hov02]{hovey2}
{\sc M. Hovey}, {\sl Cotorsion pairs, model category structures, and
representation theory}, Math.\ Z. {\bf 241} (2002), 553--592.



\bibitem[Hov04]{Hov3}
{\sc M. Hovey}, {\sl
Homotopy theory of comodules over a Hopf algebroid}, Homotopy theory: relations with algebraic geometry, group cohomology, and algebraic K-theory, 261?304,
Contemp.\ Math., {\bf 346}, Amer.\ Math.\ Soc., Providence, RI, 2004. 

\bibitem[HPS97]{HPS97}
{\sc M.\ Hovey, J.\ H.\ Palmieri and N.\ P.\ Strickland}, {\sl Axiomatic stable homotopy theory}, Mem. Amer. Math. Soc.{\bf 128}(1997), no. 610.


\bibitem[LMB00]{LMB00}
{\sc G.\ Laumon and L.\ Moret-Bailly}, {\sl Homology}, Springer--Verlag, Berlin, 2000.








\bibitem[Lur05]{Lurie}
{\sc J.\ Lurie}, {\sl Tannaka duality for geometric stacks}, preprint, \texttt{arXiv:math.AG/0412266v2}.


\bibitem[May01]{May} {\sc J.\ P.\ May}, {\sl The additivity of traces in triangulated categories.} Adv.\ in Math.
{\bf 163}(2001), 34--73.


\bibitem[Mur08]{Mur2}{\sc D.\ Murfet},
{\sl The Mock Homotopy Category of Projectives and Grothendieck Duality}, PhD
thesis,
Aust. National U. 2008.



\bibitem[Nau7]{Nau07} {\sc N.\ Naumann}, {\sl The stack of formal groups in stable homotopy theory.} Adv.\ in Math.
{\bf 215}(2007), 569--600.

\bibitem[Nee08]{neeman1}
{\sc A.\ Neeman}, {\sl The homotopy category of flat modules, and Grothendieck
duality}, Invent.\ Math. {\bf 174}(2008), 255--308.

\bibitem[Nee10]{neeman}
{\sc A.\ Neeman}, {\sl Some adjoints in homotopy categories,} Ann.\
Math. {\bf 171}(2010), 2143--2155.
 


\bibitem[Qui67]{Quillen}
{\sc D.\ Quillen}, {\sl Homotopical Algebra}, LNM {\bf 43},
Springer--Verlag, Berlin/New York, 1967.

\bibitem[Ols07]{olsson} {\sc M.\ Olsson}, {\sl Sheaves on Artin stacks.} J.\ Reine
Angew.\ Math. {\bf 603}(2007), 55-112.

\bibitem[RG71]{RG} {\sc M.\ Raynaud and L.\ Gruson},
{\sl Crit\`eres de platitude et de projectivit\'e}, Invent.\ Math. {\bf
13}(1971), 1--89.

\bibitem[Ryd13]{Rydh}
{\sc D.\ Rydh}, {\sl Noetherian approximation of algebraic spaces and stacks}, preprint, \texttt{arXiv:math.AG/0904.0227v3}.




\bibitem[Tho87]{Tho87} {\sc R.\ W.\ Thomason}, {\sl Equivariant resolutions, linearization, and Hilbert's fourteenth problem over arbitrary base schemes.} Adv.\ in Math.{\bf 65}(1987), 16--34.


\bibitem[TV08]{TV08}{\sc B.\ To\"en and G.\ Vezzosi}, 
{\sl Homotopical algebraic geometry. II. Geometric stacks and applications.},
Mem. Amer. Math. Soc.{\bf 193}(2008), no. 902.




\bibitem[Tot04]{Totaro}
{\sc B. Totaro}, {\sl The resolution property for schemes and stacks},
J. reine angew. Math. {\bf 577} (2004), 1--22.






















\end{thebibliography}
\end{document}